\documentclass[12pt]{article} 
\usepackage[utf8x]{inputenc}
\usepackage{tikz} 
\usepackage[all]{xy}
\usepackage{authblk}
\usepackage[babel=true]{csquotes}
\usepackage{amsmath,amsthm, amssymb,amsfonts}
\usepackage[hmargin=1in,vmargin=1in]{geometry}
\numberwithin{equation}{subsection} 
\xyoption{arc}
\usepackage{eucal}
\usepackage{indentfirst}
\usepackage{mathrsfs}
\usepackage{verbatim}
\usepackage{makeidx}
\usepackage{graphicx}
\usepackage{yfonts} 
\usepackage{hyperref}
\usepackage{enumitem} 
\usepackage{extpfeil} 
\usepackage{dsfont}
\usepackage[T1]{fontenc}
\newtheorem{thm}{Theorem}[section]
\newtheorem{prop}[thm]{Proposition}
\newtheorem{pdef}[thm]{Proposition-Definition}
\newtheorem{lem}[thm]{Lemma}

\newtheorem{cor}[thm]{Corollary}

\newtheoremstyle{bidule}
{3pt}
{3pt}
{}
{}
{\scshape}
{.}
{.5em}
{}
\newtheorem{df}[thm]{Definition}
\newtheorem{hypo}{Hypothesis}
\theoremstyle{definition}
\newtheorem{rmk}[thm]{Remark}
\newtheorem{ex}[thm]{Example}

\newtheorem*{term}{Terminology}
\newtheorem*{note}{Note}

\newtheorem*{warn}{Warning}

\newtheorem{nota}[thm]{Notation}

\newcommand{\C}{\mathcal{C}}
\newcommand{\Ub}{\mathcal{U}}

\newcommand{\F}{\mathcal{F}}

\renewcommand{\O}{\mathcal{O}}
\newcommand{\Ar}{\text{Arr}}

\newcommand{\D}{\mathcal{D}}
\newcommand{\Ba}{\mathcal{B}}

\newcommand{\Q}{\mathbb{Q}}

\newcommand{\Aa}{\mathcal{A}}
\newcommand{\B}{\mathscr{B}}

\newcommand{\M}{\mathscr{M}}

\newcommand{\Ja}{\mathbf{J}} 
\newcommand{\J}{\mathcal{J}} 

\newcommand{\T}{\mathcal{T}}

\newcommand{\Ea}{\mathcal{E}} 
 
\newcommand{\W}{\mathscr{W}}

\newcommand{\Nv}{\mathscr{N}}
\newcommand{\I}{\mathbf{I}}
\newcommand{\Un}{\mathbb{I}} 

\newcommand{\G}{\mathcal{G}}

\renewcommand{\S}{\mathcal{S}}

\newcommand{\Fb}{\mathbf{F}}
\newcommand{\Sim}{\mathscr{S}}

\renewcommand{\to}{\longrightarrow}
\newcommand{\ol}{\overline}
\newcommand{\ul}{\underline}

\renewcommand{\bf}{\mathbf}

\newcommand{\tx}{\text}

\renewcommand{\to}{\longrightarrow}
\DeclareMathOperator\Id{Id}

\DeclareMathOperator\Cat{\mathbf{Cat}}
\DeclareMathOperator\colim{\tx{$colim$}}
\DeclareMathOperator\limi{\tx{$lim$}}
\DeclareMathOperator\msxsu{\M_{\S}(\tx{$X$})_{\tx{$u$}}}

\DeclareMathOperator\kb{\mathbf{K}} 

\DeclareMathOperator\Lb{\tx{L}} 

\DeclareMathOperator\Ho{\mathbf{ho}} 

\DeclareMathOperator\sset{\mathbf{sSet}} 
 
\DeclareMathOperator\av{\alpha_{\downarrow_{\Id_V}}}

\DeclareMathOperator\coms{\M_\Ub[\ag]} 
\DeclareMathOperator\comsec{\M_\Ub[\ag]_{\tx{$e$}}^{\mathbf{c}}}
\DeclareMathOperator\comse{\M_\Ub[\ag]_{\tx{$e$}}}  
\DeclareMathOperator\comsep{\M_\Ub[\ag]_{\tx{$e+$}}}
\DeclareMathOperator\comsepc{\M_\Ub[\ag]_{\tx{$e+$}}^{\mathbf{c}}}  
\DeclareMathOperator\com{\ag} 
 
\DeclareMathOperator\oalg{\O\tx{-$Alg$}}
\DeclareMathOperator\ag{\textgoth{A}}
\DeclareMathOperator\mua{\M_\Ub[\ag]}

\title{Quillen-Segal objects and structures:
an overview} 
\author{Hugo V. Bacard \thanks{\textit{E-mail address}: \href{mailto:hbacard@uwo.ca}{hbacard@uwo.ca}
}}
 \affil{Western University}
\date{June 29, 2014}
\begin{document}
\maketitle
\begin{abstract}
Let  $\M$  be a combinatorial and left proper model category, possibly with a monoidal structure.  If $\O$ is either a monad on $\M$, or an operad enriched over $\M$, define a \emph{QS-algebra} in $\M$ to be a  weak equivalence $\F: s(\F) \xrightarrow{\sim}t(\F)$ such that the target $t(\F)$ is an $\O$-algebra in the usual sense.  A classical $\O$-algebra is a QS-algebra supported by an isomorphism $\F$. A \emph{QS-structure} $\F$ is also a weak equivalence such that $t(\F)$ has a \emph{structure}, e.g \emph{Hodge, twistorial, schematic, sheaf}, etc.  We build a homotopy theory of these objects and compare it with that of usual $\O$-algebras/structures.  Our results rely on Smith's theorem on left Bousfield localization for combinatorial and left proper model categories. These ideas are derived from the theory of co-Segal algebras and categories.
\end{abstract}
\setcounter{tocdepth}{1}
\tableofcontents

\section{Introduction}
This paper is originally motivated by the desire to have Quillen model structure on weak algebras for a given monad or an operad $\O$. And we wanted to do this under minimal hypotheses on the ground model category $\M$. But the big picture that comes out is actually related to classical results of Quillen \cite{Quillen_HK}, Thomason \cite{Thomason_cat}, Jardine \cite{Jardine_simpre}, Joyal \cite{Joyal_simpsh} and many others.

 The idea is to start with a functor  $\Ub: \ag \to \M$ such as a forgetful functor $\Ub:  \oalg(\M) \to \M$ from algebras to underlying objects; or a forgetful functor  $$\Ub:C(\Q MHS) \to C(\Q)$$ from complexes of rational Mixed Hodge Structures to complexes of rational vector spaces.

Other relevant examples include the Grothendieck-Quillen-Thomason example of the  \emph{nerve functor} $\Cat \to \sset$, from categories to simplicial sets; or a (higher) \emph{fiber functor} in Tannakian  settings; or an inclusion of $Sh(\C) \hookrightarrow sPresh(\C)$ of (pre)sheaves to simplicial presheaves on a Grothendieck site.\\

Given such $\Ub$, we can form the \emph{comma category} $(\M\downarrow\Ub)$ which we described hereafter for the reader's convenience. The objects of $(\M\downarrow\Ub)$ are morphisms $\F: a \to \Ub(b)$. Given another object $\G: a' \to \Ub(b')$, a morphism $\sigma:\F \to \G$ is given by a morphism $i: a \to a'$ in $\M$ and a morphism $j: b \to b'$ in $\ag$ such that we have a commutative square in $\M$:
\[
\xy
(0,18)*+{a}="W";
(0,0)*+{\Ub(b)}="X";
(30,0)*+{\Ub(b')}="Y";
(30,18)*+{a'}="E";
{\ar@{->}^-{\Ub(j)}"X";"Y"};
{\ar@{->}^-{\F}"W";"X"};
{\ar@{->}^-{i}"W";"E"};
{\ar@{->}^-{\G}"E";"Y"};
\endxy
\]

There is a forgetful functor $|-|: (\M\downarrow \Ub) \to \ag$ that corresponds to the target functor;  and  we also have a source functor $s: (\M\downarrow \Ub) \to \M$. We are actually interested in the full subcategory  $(\W \downarrow \Ub) \subseteq (\M\downarrow \Ub)$, where $\W \subseteq \M$ is the subcategory of weak equivalences in $\M$. \\

Our \emph{QS-objects} (Quillen-Segal objects), are precisely the objects of $(\W \downarrow \Ub)$ consisting of weak equivalences $\F: a \xrightarrow{\sim} \Ub(b)$. The goal of the paper is to have a model structure on $(\M\downarrow \Ub)$ such that fibrant objects belongs to $(\W \downarrow \Ub)$. We also want to analyze the functor $(\M \downarrow \Ub) \to \ag$, when there is a model structure on $\ag$.\\

To achieve this, we use some classical techniques that go back to works of Jardine \cite{Jardine_simpre} and Joyal \cite{Joyal_simpsh}. A classical result  of  Bousfield-Friedlander \cite{Bous_Fried} also provides  some criterions to get a model structure in this setting (see also \cite{Stanculescu_BF}).   Simpson \cite{Simpson_HTHC} later refers to this process as a  \emph{weak (monadic) projection}.  It starts with a key observation that  there is an endofunctor $\Sim: (\M\downarrow \Ub) \to (\M\downarrow \Ub)$ whose image is in $(\W\downarrow \Ub)$.  Indeed, given $\F: a \to \Ub(b)$, we can use the axiom of the model category $\M$ and factor $\F$ as a cofibration followed by a trivial fibration:
$$a \xrightarrow{\F} \Ub(b)= a \hookrightarrow m \xtwoheadrightarrow[\sim]{\Sim(\F)} \Ub(b).$$

This factorization can be chosen to be functorial and even \emph{algebraic} if $\M$ is an \emph{algebraic model structure} in the sense of Riehl \cite{Riehl_algcat}. After this, one defines a map $\sigma: \F \to \G$ to be a weak equivalence in $(\M\downarrow \Ub)$ if the component $m \to m'$ in $\Sim(\sigma)$ is a weak equivalence in $\M$. And by $3$-for-$2$ it also forces the component $\Ub(b) \to \Ub(b')$ to be a weak equivalence. Put differently, this simply means that $\Sim(\sigma)$ is a commutative square of weak equivalences.\\

To get a model structure,  we need a nice category $\ag$, typically a locally presentable category. We also need a left adjoint   $\Fb: \M \to \ag$ to the functor $\Ub$; and we assume that $\Ub$ is faithful(=injective on morphisms). The ground model  category $\M$ is required to be combinatorial and left proper.  Our main result is the content of  Theorem \ref{main-thm-quillen-equiv}.\\

These hypotheses are enough to build a homotopy theory on our \emph{QS-objects}. For example in the case of algebras, we don't require further axioms such as the \emph{monoid axiom} of \cite{Sch-Sh-Algebra-module}. It also extends the work done in \cite{Bacard_dgcom}, about commutative dg-algebras in characteristic $p>0$. If we let $\ag$ be $\oalg(\M)$  with the free-forgetful adjunction  $\Ub: \oalg(\M) \leftrightarrows \M: \Fb$, the identity morphisms of algebras give an  embedding $$\iota: \oalg(\M) \hookrightarrow (\M\downarrow \Ub),$$
whose left adjoint is precisely the target functor:
$$|-|: (\M\downarrow \Ub) \to \oalg(\M).$$ 

If the projective (usual) model structure on $\oalg(\M)$ exists we will show that the adjunction $|-| \dashv \iota$ is a Quillen equivalence.\\

In general, if $\Ub: \ag \to \M$ reflects the weak equivalences i.e., $j$ is a weak equivalence in $\ag$ as soon as $\Ub(j)$ is a weak equivalence in $\M$, then the functor $|-|: (\M\downarrow \Ub) \to \ag$ sends weak equivalences to weak equivalences. And in particular it descends to the respective Gabriel-Zisman localizations (see \cite{Gabriel_Zisman}). If $\Ub$ creates the weak equivalences in $\ag$ then the inclusion $\iota: \ag \to (\M\downarrow \Ub)$  will preserve and reflect the weak equivalences. In particular we also have an induced functor between the respective homotopy categories.\\

 In this situation even if $\ag$ is not a model category, it can be shown under mild conditions,  that the inclusion $\iota: \ag \to (\M\downarrow \Ub)$ induces an equivalence between the respective homotopy categories.
\subsection*{Related works} 
Model structures on comma categories have been discussed  by Stanculescu \cite[Sec. 6.5]{Stanculescu_multi}, but the weak equivalence he considered are the natural weak equivalences, which turn out to be too strong for our purposes.

The general philosophy of this paper is also close to that of Hollander \cite{Hollander_stack}, Hirschowitz-Simpson \cite{HS}, Jardine \cite{Jardine_simpresh}, Joyal-Tierney \cite{Joy-Tier},  Stanculescu \cite{Stanculescu_stack} and many others.\\

We can apply our main result (Theorem \ref{main-thm-quillen-equiv}) when  $\ag= \M$ and  $\Ub= \Id_{\M}$ is the identity functor $\Id: \M \to \M$. In this case $ (\M\downarrow \Ub)$ is the category $\Ar(\M)$ of morphisms in $\M$ and $|-|: \Ar(\M) \to \M$ is the target functor. Our theorem gives a Quillen equivalence 
$$|-|: \Ar(\M) \leftrightarrows \M: \iota.$$

If we apply our result to the nerve functor $\Nv: \Cat \to \sset$ twisted by $Ex_2$ (see \cite{Thomason_cat}), and its left adjoint, we get a Quillen equivalence $$|-|: (\sset \downarrow \Cat)  \leftrightarrows \Cat: \iota,$$
where in $\Cat$ we use the Thomason model structure. This gives also a model for the homotopy theory of spaces\footnote{There are some bound issue here, in order to get a combinatorial model structure. These have been taken care of in \cite{Low_heart}} ($\sset$). If we don't use $Ex_2$, we no longer have Thomason's model structure on $\Cat$ but we do have a model structure on $(\sset \downarrow \Cat)$ with the same weak equivalences, so we still have the same homotopy categories.  \\

There are many cases where $\ag$ has a model structure that is right-induced through an adjunction $\Ub: \ag \leftrightarrows \M: \Fb$. Sometimes, it happens that the model category $\ag$ is not left proper even if $\M$ is. By a result of Dugger \cite{Dugger, Dugger_rep}, we know  that if $\ag$ is a combinatorial model category, then it's Quillen equivalent to another one that is left proper. In our case the category $(\M\downarrow \Ub)$ is an example of such \emph{Dugger model} for $\ag$. 

In fact our method is close to that of Dugger. Basically we truncate $\Delta^{op}$ by only keeping the the interval $[0 \to 1]$, which is the opposite of the unique morphism $1 \to 0$ in $\Delta$. This idea is motivated by the theory of co-Segal algebras \cite{Bacard_dgcom}.

We have a similar observation for a theorem of  Rezk \cite{Rezk_model} on algebras of a simplicial algebraic theory $\T$.  But our theorem is not as satisfactory as   Rezk's since  the model structure we get here for weak $\T$-algebras is a priori only left proper and not necessarily right proper. But one advantage is probably that we don't change the algebraic theory $\T$.

It seems that the same thing happens with  Bergner's rigidification theorem \cite{Bergner_rigid},  which is based on a earlier result of Badzioch \cite{Badzioch_rigid}. In fact our theorem in Bergner's setting could be seen as a \emph{co-Segal} analogue.  This analogy is justified by the fact that, based on this rigidification result, Bergner \cite{Bergner_inf} shows that Segal categories are Quillen equivalent to simplicial categories. \\

We also include a short observation on the left properness of the category of algebras over an operad or a monad (see Proposition \ref{prop-left-prop} and Corollary \ref{left-prop-alg}).\\

The present material can be extended to situations where $\Ub: \com \to \M$ is a monoidal functor. As pointed out by Hovey \cite{Hov-model} (after Smith),  we should put a monoidal structure  on $(\M\downarrow \Ub)$ with the \emph{pushout product}. The pushout product doesn't affect the targets, therefore we will still have a monoidal functor 
$$(\M\downarrow \Ub) \to \com.$$

\begin{note}
This paper aims to expose first the general idea, and we apologize in advance because most of the results are given without proof. Another more elaborated version of this paper will contain much more details as well as some applications.
\end{note}

\begin{nota}
For simplicity, we write in the main text $\mua$ instead of  $(\M \downarrow \Ub)$. We think of this category as a category of objects of $\M$ with \emph{coefficients in $\ag$} (rel. to  $\Ub$). 
\end{nota}
  
\section{Definition and Properties}
 
\begin{hypo}\label{hypo-m}
In this paper we assume that
\begin{enumerate}
\item $\M$ is either a symmetric monoidal closed category $\M=(\ul{M}, \otimes,I)$,   such that the underlying  category $\ul{M}$ is  a combinatorial and left proper model category; 
\item or simply a combinatorial and left proper model category that is not necessarily part of a monoidal structure $\otimes$.
\item In particular $\M$ is locally presentable.
\item We  consider a faithful functor $\Ub: \com \to \M$ that is also a right adjoint between locally presentable categories. 
\end{enumerate} 
\end{hypo}

\begin{warn}
For simplicity we will freely identity $\M$ and $\ul{M}$. 
\end{warn}

Let $\O$ be either a monad on $\M$ or an operad enriched over $\M$.
\begin{df}
An $\O$-QS-algebra in $\M$ is a weak equivalence $\F: s(\F) \xrightarrow{\sim} t(\F)$ such that the target $t(\F)$ is a usual $\O$-algebra.\\

Let $\G: s(\G) \to t(\G)$ be another $\O$-QS-algebra. A map $\sigma: \F \to \G$ of QS-algebras is a map in the arrow-category i.e., a commutative square such that the component $t(\F) \to t(\G)$ is a map of $\O$-algebras.
\end{df}
\begin{rmk}
This notion is inspired by the theory of co-Segal algebras. In the homotopy theory of co-Segal algebras, every co-Segal algebra is equivalent to a QS-algebra. Everything that will be said here can be generalized to categories.
\end{rmk}

\begin{df}
Let $\ag \subset \M$ be a subcategory of objects with a \emph{structure}. An \emph{$\ag$-QS-object} is a weak equivalence $\F: s(\F) \xrightarrow{\sim} t(\F)$ such that $t(\F) \in \ag$.  A morphism $\sigma: \F \to \G$ of  $\ag$-QS-objects is a commutative square such that the component $t(\F) \to t(\G)$ is a morphism in $\ag$. 
\end{df}

\begin{rmk}
Following Grothendieck's spirit, we consider a more general situation where we replace the inclusion $\ag \subset \M$ by a functor $\Ub: \ag \to \M$ that is injective on morphisms (faithful functor).  We're going to generalize the previous definitions to this setting since both concepts fit in this formalism. Indeed, we recover the case of QS-algebras by letting $\ag$ be the category $\oalg$, of $\O$-algebras in $\M$ with the forgetful functor $\Ub: \oalg \to \M$.
\end{rmk}
Before going further, let's introduce some notation that will simplify the discussion. 

\begin{nota}
\begin{enumerate}
\item Let's denote by $\Un= \{0 \to 1\}$ the \emph{walking-morphism category}. As usual a morphism $\F$ in $\M$ is the same thing as a functor $\F:\Un \to \M$.
\item We will therefore denote by $\M^{\Un}=\Ar(\M)$, the category of morphisms of $\M$. And from now on, we will write $\F: \F_0 \to \F_1$.
\item We will denote by $s: \M^{\Un} \to \M$ and $t : \M^{\Un} \to \M$ the usual source and target functors. They correspond respectively to the evaluation at $0$ and $1$.
\item Let $\iota_\M:\M \to \M^{\Un}$ be the natural embedding that takes an object $m$ to $\Id_{m}$; it takes a morphism $f: m \to m'$ to the morphism $[f]: \Id_m \to \Id_{m'}$ of $\M^{\Un}$, whose components are both equal to $f$.
\item It's easily seen that $t: \M^{\Un} \to \M$ is left adjoint to $\iota_\M$. In fact, $t$ is the colimit functor when we regard a morphism as a diagram $\Un \to \M$.
\end{enumerate}
\end{nota}

To build a homotopy theory of our objects,  we proceed exactly as for Segal categories, co-Segal categories, quasicategories, Segal spaces, etc. We enlarge our definition from QS-algebras to prealgebras. A QS-algebra (resp. QS-structure) is a prealgebra (prestructure) that satisfies the analogue of the \emph{co-Segal conditions} (see \cite{Bacard_twec,Bacard_dgcom}). This idea of enlarging our class of objects goes back to Jardine \cite{Jardine_simpre} and Joyal \cite{Joyal_simpsh}, and has been widely used in the literature (see for example \cite{Simpson_HTHC})\begin{df}
Let $\M$ be as above and let $\Ub: \ag \to \M$ be a faithful functor that reflects isomorphisms. 
\begin{enumerate}
\item An \emph{$\Ub$-preobject} in $\M$ is a morphism $\F: \F_0 \to \F_1 \in \M^{[\Un]}$ such that $\F_1$ is in the image of $\Ub$.
\item An $\Ub$-preobject is an \emph{$\Ub$-QS-object} if $\F$ is a weak equivalence in $\M$. 
\item A map $\sigma: \F \to \G$ of $\Ub$-preobjects is a map in $\M^{\Un}$ such that the component $\sigma_1: \F_1 \to \G_1$ is a map in the image of $\Ub$.
\end{enumerate}
Denote by $\mua$ the category of $\Ub$-preobjects and morphisms between them. There is a forgetful functor 
$$\Pi: \mua \to \M^{\Un}.$$
Taking the source and target gives two functors:
$$s: \mua \to \M, \quad t: \mua \to \M.$$ 
The functor $t$ factors through an obvious functor $|-|: \mua \to \ag$ and we have a commutative triangle:
\begin{equation}\label{triang1}
\xy
(0,15)*+{\mua}="W";
(0,0)*+{\ag }="X";
(20,0)*+{\M}="Y";
{\ar@{->}^-{\Ub}"X";"Y"};
{\ar@{->}^-{|-|}"W";"X"};
{\ar@{->}^-{t}"W";"Y"};
\endxy
\end{equation}
The composite $\ag \xrightarrow{\Ub} \M \xrightarrow{ \iota_\M} \M^{\Un}$ extends to a functor $\iota: \ag \hookrightarrow \mua$ that identifies $\ag$ as a full subcategory of $\mua$. We also have a commutative triangle:
\begin{equation}\label{triang2}
\xy
(0,15)*+{\mua}="W";
(0,0)*+{\ag }="X";
(20,0)*+{\M}="Y";
{\ar@{->}^-{U}"X";"Y"};
{\ar@{<-}^-{\iota}"W";"X"};
{\ar@{->}^-{s}"W";"Y"};
\endxy
\end{equation}
\end{df}

\begin{term}
By analogy, we shall call the functor $|-|: \mua \to \ag$ an \emph{augmentation} or an \emph{anchor}.
\end{term}

\begin{rmk}
\begin{enumerate}
\item Thanks to the axioms of the model category $\M$, there is an endofunctor $\Sim: \mua \to \mua$ that takes an $\Ub$-preobject $\F$ to an $\Ub$-QS-object $\Sim(\F)$. Indeed it suffices to factor $\F$ as a cofibration followed by a trivial fibration:
$$\F_0 \xrightarrow{\F} \F_1= \F_0 \hookrightarrow m \xtwoheadrightarrow[\sim]{\Sim(\F)} \F_1.$$
\item As we see this operation doesn't change the target $\F_1$ that possesses the \emph{structure}. This factorization provides a tautological map $\eta: \F \to \Sim(\F)$ in $\mua$. We will show later that this map is a weak equivalence in a suitable model structure (Bousfield localization). 
\item It's also a key observation that if we regard $\F_1$ as an object of $\mua$ through the inclusion $\iota: \ag \to \mua$, then $\F$ is simultaneously an object of $\mua$ and a morphism $\F \to \iota(\F_1)$ in $\mua$. Similarly $\Sim(\F)$ is simultaneously an object and a morphism $\Sim(\F) \to \iota(\F_1)$ which is a level-wise weak equivalence in $\M$.
\item In the end we will show later that the morphism $\F \to \iota(\F_1) $ is a weak equivalence in a suitable model structure.
\end{enumerate}
\end{rmk}

Our fist step toward the homotopy theory is the following definition.
\begin{df}\label{df-easy-weak}
Say that a map $\sigma: \F \to \G$ in $\mua$ is an \emph{easy weak equivalence} (resp. fibration) if the image $s(\sigma): \F_0 \to \G_0$ is a weak equivalence (resp. fibration) in $\M$. 

\end{df}

Our first result is the following theorem whose proof will be given at the end of the next subsection. Basically it simply says that, under some hypotheses on $\Ub$ and $\ag$, there is a \emph{right induced} model category on $\mua$ using the source functor (and it's left adjoint).
\begin{thm}\label{easy-model-mua}
Let  $\M$ be a combinatorial and left proper model category.  Let $\Ub: \ag \to \M$ be a faithful functor that reflects isomorphisms.  Assume also that $\Ub$ has a left adjoint  $\Fb: \M \to \ag$ and that $\ag$ is locally presentable. 
\begin{enumerate}
\item Then there is a combinatorial model structure on $\mua$ which may be described as follows.
\begin{itemize}[label=$-$]
\item A map $\sigma :\F \to \G$ is a weak equivalence (resp. fibration) if it's an easy weak equivalence (resp. fibration). 
\item A map  $\sigma :\F \to \G$ is a cofibration if it has the LLP against any map that is a weak equivalence and a fibration. 
\end{itemize}
\item This model structure is also left proper.

\item We will denote by $\mua_e$ this model category. The  source functor
$$ s: \mua_e \to  \M$$
is a right Quillen functor.
\item If there is a model structure on $\ag$ such that $\Ub: \ag \to \M$ is right Quillen,  then the functor $|-|: \mua_e \to \ag$ is a left Quillen functor.
\end{enumerate} 
\end{thm}

This theorem is not exactly what we want. Having a combinatorial and left proper model category allows us to perform the machinery of left Bousfield localization, using Smith's theorem (see \cite{Barwick_localization}). 

\begin{ex}
Examples of such situation include the category of complexes in a Grothendieck abelian category. Another important example is the category of complexes of rational mixed Hodge structures together with the forgetful functor that lands in the category $C(\Q)$ of complexes of rational vector spaces. It seems that the work of Cirici-Guillén \cite{Cirici_Gui} can be performed using Quillen model categories.
\end{ex}

\subsection{Limits, colimits and adjunction}

\subsubsection{Adjunctions}
The following result follows from a direct checking. 
\begin{prop}
Assume that $\Ub: \ag \to \M$ is faithful functor that possesses a left adjoint $\Fb: \M \to \ag$. Then we have an adjunction:
 $$|-|: \mua \leftrightarrows  \ag: \iota,$$ where $|-|$ is a left adjoint and $\iota$ is a right adjoint. In other words, the functor $\iota: \ag \hookrightarrow \mua$ exhibits $\ag$ as a full reflective subcategory.  $\qed$
\end{prop}

\begin{lem}
Let $\M$ be as before and let $\Ub: \ag \to \M$ be a faithful functor that has a left adjoint $\Fb: \M \to \ag$.
\begin{enumerate}
\item The functor $\Pi: \mua \to \M^{\Un}$ has a left adjoint $\Gamma: \M^{\Un} \to \mua$. 
\item The source functor $s: \mua \to \M$ has a left adjoint $\Fb_s: \M \to \mua$. And $s$ preserves colimits and in particular it preserves pushouts
\item If $\ag$ has a terminal object, then $s$ has also a right adjoint $taut_s: \M \to \mua$. In particular $s$ also preserves limits.
\item There is a commutative triangle of left adjoints, and the triangle \eqref{triang2} is formed by the respective right adjoints:

\begin{equation}\label{triang3}
\xy
(0,15)*+{\mua}="W";
(0,0)*+{\ag }="X";
(20,0)*+{\M}="Y";
{\ar@{<-}^-{\Fb}"X";"Y"};
{\ar@{->}_-{|-|}"W";"X"};
{\ar@{<-}^-{\Fb_s}"W";"Y"};
\endxy
\end{equation}
\end{enumerate}
\end{lem}

\begin{proof}
Let $\eta: m \to \Ub\Fb(m)$ be the unit of the adjunction $\Fb \dashv \Ub$. Given $f: f_0 \to f_1$  in $\M^{\Un}$, define $\Gamma(f)$ as the composite
$$f_0 \to f_1 \xrightarrow{\eta} \Ub\Fb(f_1).$$

It's easily seen that this defines a left adjoint to $\Pi: \mua \to \M^{\Un}$. This proves the first assertion.\\

The left adjoint $\Fb_s$ is just the composite of left adjoints:
$$\M \xrightarrow{\iota_\M} \M^{\Un} \xrightarrow{\Gamma} \mua.$$
Explicitly, $\Fb_s(m)$ is just the unit $\eta_m: m \to \Ub\Fb(m)$.
This gives the second assertion since left adjoint always preserves colimits.

For assertion $(3)$ we proceed as follows. Let $\ast$ be the terminal object in $\M$.  Since $\Ub$ is a right adjoint, it must send terminal object to terminal object. Given $m \in \M$ define $taut_s(m) \in \mua$ to be the unique morphism $m \to \ast$ going to the terminal object. Clearly this defines a right adjoint to $s$ as the reader will check. Finally right adjoint always preserves limits.

The last assertion is clear.
\end{proof}

\begin{rmk}\label{rmk-adj-gamma}
It's important to notice that the left adjoint $\Gamma: \M^{\Un} \to \mua$ doesn't  change the source. In other word, given $f: f_0 \to f_1$ we have $\Gamma(f)_0=f_0$.  And if $\chi: f \to g$ is a morphism, then $\Gamma(\chi)_0= \chi_0$.
\end{rmk}

\subsubsection{Limits and Colimits}
\begin{note}
We list below some technical results on locally presentable categories. Good references on the subject include \cite{Adamek-Rosicky-loc-pres}, \cite{Chorny_Rosi}, \cite{Low_heart}. The last reference is very useful for our discussion because the author deeply treats in \cite{Low_heart}, the size issue for combinatorial model categories. 

\end{note}
Given a diagram $D: \J \to \mua$, denote by $|D|: \J \to \ag$ the projection in $\ag$ and denote by $s(D): \J \to \M$ the other projection in $\M$.  Denote by $\colim |D|$ and $\colim s(D)$ the respective colimits. Similarly denote by $\limi |D|$ and $\limi s(D)$ the respective limits. We implicitly assume the size of $\J$ is such that both objects exist simultaneously in $\M$ and $\ag$.

\begin{lem}\label{lem-limit-colimit}
Let $\Ub: \ag \leftrightarrows \M: \Fb$ be an adjunction between locally presentable categories, where $\Ub$ is faithful and reflects isomorphisms. Then with the previous notation the following hold. 
\begin{enumerate}
\item There is a unique map  $$\colim s(D) \to \Ub[\colim |D|]$$ 
in $\M$ induced by the universal property of the colimit-object  $\colim s(D)$. This morphism is the colimit of the diagram  $D: \J \to \mua$.
\item Similarly, there is a unique map  $$\limi s(D) \to \Ub[\limi |D|]$$ in $\M$ induced by the universal property of the limit $\Ub(\lim|D|)$, since $\Ub$ preserves limits. This morphism is the limit of the diagram  $D: \J \to \mua$. $\qed$
\end{enumerate}
\end{lem}

The following result can be proved directly, but we refer the reader to \cite{Low_heart} for a much more general proof. 
\begin{prop}
Let  $\Ub: \ag \leftrightarrows \M: \Fb$ be an adjunction between locally presentable categories, where $\Ub$ is faithful and reflects isomorphism.

Then the category $\mua= (\M \downarrow \Ub)$ is also locally presentable.
\end{prop}

\begin{rmk}
The requirement that $\Ub$ reflects isomorphisms might not be necessary in most cases. This hypothesis is suggested by the case of algebras. 
\end{rmk}
\section{Right-induced model structure}
Consider the adjunction $s: \mua \leftrightarrows \M: \Fb_s$. Recall that in  Definition \ref{df-easy-weak}, we call a  morphism $\sigma: \F \to \G$,  an \emph{easy} weak equivalence (resp. fibration) if $s(\sigma): \F_0 \to \G_0$ is a weak equivalence (resp. fibration) in $\M$. 
  
We would like to show using the well know result \cite[Theorem 11.3.2]{Hirsch-model-loc}, that there is a model structure on $\mua$ with the above definition of weak equivalences and fibrations.\\

First we have a direct consequence of Remark \ref{rmk-adj-gamma}:
\begin{prop}\label{prop-cof-unchange}
If $\chi: f \to g$ is a map in the arrow-category $\M^{\Un}$ such that $s(\chi): f_0 \to g_0$ is a (trivial) cofibration (resp. weak equivalence), then  $s(\Gamma(\chi))= s(\chi)$ is also a (trivial) cofibration (resp. weak equivalence). 
\end{prop}

The following lemma is a key ingredient. The proof is  based on how we compute colimits in $\mua$. Indeed we saw before that the source-functor $s: \mua \to \M$ preserves colimits. 
\begin{lem}\label{lem-pushout-important}
Let $\theta: \F \to \G$ be a map in $\mua$ and let $D$ be a pushout square in $\mua$:
 \[
 \xy
(0,18)*+{\F}="W";
(0,0)*+{ \G}="X";
(30,0)*+{\Ba}="Y";
(30,18)*+{\Aa}="E";
{\ar@{->}^-{\ol{\sigma}}"X";"Y"};
{\ar@{->}^-{\theta}"W";"X"};
{\ar@{->}^-{\sigma}"W";"E"};
{\ar@{->}^-{\varepsilon}"E";"Y"};
\endxy
\]
Then the image of $D$ by the $s$ is the pushout square in $\M$:
 \[
 \xy
(0,18)*+{\F_0}="W";
(0,0)*+{ \G_0}="X";
(30,0)*+{\Ba_0}="Y";
(30,18)*+{\Aa_0}="E";
{\ar@{->}^-{\ol{\sigma}_{0}}"X";"Y"};
{\ar@{->}^-{\theta_0}"W";"X"};
{\ar@{->}^-{\sigma_0}"W";"E"};
{\ar@{->}^-{\varepsilon_0}"E";"Y"};
\endxy
\]
In particular:
\begin{itemize}[label=$-$]
\item The map $\varepsilon_{0}: \Aa_0 \to \Ba_0$ is a (trivial) cofibration if $\theta_0$ is a (trivial) cofibration;
\item Since $\M$ is left proper, the map $\ol{\sigma}_{0}$ is a weak equivalence if $\theta_0$ is a cofibration and if $\sigma_0$ is a weak equivalence. 
\end{itemize}
\end{lem}

\paragraph{A note about Theorem \ref{easy-model-mua}}

With the previous lemma at hand, Theorem \ref{easy-model-mua} is a direct application of  \cite[Theorem 11.3.2]{Hirsch-model-loc}. 
\begin{proof}[Proof of Theorem \ref{easy-model-mua}]
Let $\I$ and $\Ja$ be respectively the generating sets of cofibrations and trivial cofibrations in $\M$. For $f \in \Ja$, clearly the source of $\Fb_s(f)$ is just $f$ itself. Lemma \ref{lem-pushout-important} gives the required condition in  \cite[Theorem 11.3.2]{Hirsch-model-loc}. \\

The other assertion of Lemma \ref{lem-pushout-important} gives the left properness of the model structure.

The set $\Fb_s(\I)$ is a set of generating cofibrations and the set $\Fb_s(\Ja)$ is a set of generating trivial cofibrations.

\end{proof}

\begin{rmk}
As mentioned in the introduction, we can apply Theorem \ref{easy-model-mua} to the identity $\Ub: \M \to \M$. This gives an easy model structure on $\M^{\Un}$ that will be denoted by $\M^{\Un}_e$ . It's clear that the adjunction 
$$\Pi: \mua_e \to \M^{\Un}_e: \Gamma $$
is a Quillen adjunction where $\Gamma$ is left Quillen.
\end{rmk}

\section{Generating sets for the Bousfield localizations}\label{section-loc-set}

\begin{nota}\label{notation-set-localization}
If $\alpha: U \to V$ is a morphism of $\M$, we will denote by 
$$\alpha_{\downarrow_{\Id_V}}: \alpha \to \Id_V,$$ 
the morphism in the arrow category $\M^{\Un}$ which is identified with the following commutative square. 
\[
\xy
(0,18)*+{U}="W";
(0,0)*+{V}="X";
(30,0)*+{V}="Y";
(30,18)*+{V}="E";
{\ar@{->}^-{\Id}"X";"Y"};
{\ar@{->}^-{\alpha}"W";"X"};
{\ar@{->}^-{\alpha}"W";"E"};
{\ar@{->}^-{\Id}"E";"Y"};
\endxy
\]
\end{nota}
\begin{rmk}\label{rmk-projec-com}
Following our previous notation, the component $(\av)_0$ is exactly the map $\alpha:U \to V$ and the component $(\av)_1$ is the identity $\Id_V$ (which is a wonderful isomorphism). Then by definition of $\Gamma$ and $|-|$ we know that the image $|\Gamma(\av)|$ in $\ag$ is the identity $\Fb(\Id_{\alpha_1})=\Fb(\Id_{V})$ which is an isomorphism in $\ag$.
\end{rmk}


\begin{df}
 Let $\I$ be set of generating cofibrations of $\M$.
\begin{enumerate}
\item Define the \emph{localizing set} for $\mua$ as
$$\kb(\I):=  \{\Gamma(\av); \quad \alpha \in \I  \}.$$
\item Let $\ast$ be the coinitial (or terminal) object of $\coms$ and let  $\sigma$ be map in $\coms$. Say that an object $\F \in \msxsu$ is $\sigma$-injective if the unique map $\F \to \ast$ has the RLP with respect to $\sigma$. 
\item Say that $\F$ is $\kb(\I)$-injective if $\F$ is $\sigma$-injective for all $\sigma \in \kb(\I)$.
\end{enumerate}
\end{df}

\subsection{Characteristics of the set $\kb(\I)$}
The following proposition is not hard, one simply needs to write down everything.
\begin{prop}\label{lem-lifting-inject}
Let $\theta=(f,g): \alpha \to p$ be a morphism in $\M^{\Un}$ which is represented by the following commutative square.

\[
\xy
(0,20)*+{U}="A";
(20,20)*+{X}="B";
(0,0)*+{V}="C";
(20,0)*+{Y}="D";
{\ar@{->}^{f}"A";"B"};
{\ar@{->}_{\alpha}"A";"C"};
{\ar@{->}^{p}"B";"D"};
{\ar@{->}^{g}"C";"D"};
\endxy
\]  

Then the following are equivalent. 
\begin{itemize}[label=$-$]
\item There is a lifting in the commutative square above i.e there exists $k: V \to X$  such that: $k \circ \alpha =f$, $p \circ k=g$.
\item There is a lifting in the following square of $\M^{\Un}$.
\[
\xy
(0,20)*+{\alpha}="A";
(20,20)*+{p}="B";
(0,0)*+{\Id_V}="C";
(20,0)*+{\ast}="D";
{\ar@{->}^{\theta}"A";"B"};
{\ar@{->}_{\av}"A";"C"};
{\ar@{->}_{}"C";"D"};
{\ar@{->}_{}"B";"D"};
\endxy
\]  

That is, there exists $\beta=(k,l): \Id_V \to p$ such that $\beta \circ \av= \theta$. 
\end{itemize}
\end{prop}

Thank to this proposition and the fact that trivial fibrations in $\M$ are the $\I$-injective maps, we can establish by adjointness the following result. 
\begin{lem}\label{k-inj-cosegal}
Let $\F$ be an object of $\coms$. Then with the previous notation, the following hold.
\begin{enumerate}
\item $\F$ is $\kb(\I)$-injective if and only if $\F$ underlies a trivial fibration in $\M$. In particular if $\F$ is  $\kb(\I)$-injective, then $\F$ is a $\Ub$-QS-object.
\item Every object $\F \in \com$  is $\kb(\I)$-injective.
\end{enumerate}
\end{lem}

\begin{proof}
If $\F$ is $\kb(\I)$-injective, by definition, $\F$ is $\Gamma(\av)$-injective for all generating cofibration $\alpha$ in $\M$. And by adjointness we find that $\F$ is $\av$-injective. Thanks to the previous proposition, this is equivalent to saying that any lifting problem defined by $\alpha$ and $\F$ has a solution.

 Consequently $\F$ is $\kb(\I)$-injective if and only if  $\F$ has the RLP with respect to all maps in $\I$, if and only if  $\F$ is a trivial fibration as claimed. This proves Assertion $(1)$. 

Assertion $(2)$ is a corollary of Assertion $(1)$ since an object of $\com$ is supported by an identity, and every identity is a trivial fibration in $\M$. 
\end{proof}

\section{Augmented model structure}
The discussion that follows is a direct adaptation of what we did for co-Segal categories and algebras (see \cite{Bacard_lpcosegx}).
\begin{nota}
\begin{enumerate}
\item Denote by $\I_{\coms}= \Fb_s(\I)$ the set of generating cofibrations in  $\comse$.
\item Denote $\I_{\coms}^+$ the set $\I_{\comse} \sqcup \kb(\I)$
\end{enumerate}
\end{nota} 

As mentioned in Remark \ref{rmk-projec-com}, for any $\sigma$ in  $\I_{\coms}^+$, the component $s(\sigma)$ is a cofibration in $\M$. \\

We use Smith's recognition Theorem for combinatorial model categories (see for example Barwick \cite[Proposition 2.2]{Barwick_localization}). This theorem gives the possibility to  construct a combinatorial model category out of two data consisting of a class $\W$ of morphisms whose elements are called \emph{weak equivalences}; and a set $\I$ of \emph{generating cofibrations}. \\

We actually use Lurie's version \cite[Proposition A.2.6.13]{Lurie_HTT}. This version asserts that the resulting combinatorial model structure is automatically left proper. We recall this theorem hereafter with the same notation as in Lurie's book. 
\begin{prop}\label{Smith-Lurie}
Let $\bf{A}$ be a presentable category. Suppose we are given a class $W$ of morphisms of A, which we will call weak equivalences, and a (small) set $C_0$ of morphisms of $\bf{A}$, which we will call generating cofibrations. Suppose furthermore that the following assumptions are satisfied:
\begin{itemize}
\item[$(1)$] The class $W$ of weak equivalences is perfect (\cite[Definition A.2.6.10]{Lurie_HTT}). 
\item[$(2)$] For any diagram
\[
\xy
(0,15)*+{X}="A";
(20,15)*+{Y}="B";
(0,0)*+{X'}="C";
(20,0)*+{Y'}="D";
(0,-15)*+{X''}="X";
(20,-15)*+{Y''}="Y";
{\ar@{->}^-{f}"A";"B"};
{\ar@{->}_-{}"A";"C"};
{\ar@{->}^-{}"B";"D"};
{\ar@{->}^-{}"C";"D"};
{\ar@{->}^-{}"X";"Y"};
{\ar@{->}^-{g}"C";"X"};
{\ar@{->}^-{g'}"D";"Y"};
\endxy
\] 

in which both squares are coCartesian (=pushout square), $f$ belongs to $C_0$, and $g$ belongs $W$, the map $g'$ also below to $W$.
\item[$(3)$] If $g: X \to Y$ is a morphism in $\bf{A}$ which has the right lifting property with respect to every morphism in $C_0$, then $g$ belongs to $W$.
\end{itemize}

Then there exists a left proper combinatorial model structure on $\bf{A}$ which
may be described as follows:
\begin{itemize}
\item[$(C)$] A morphism $f : X \to  Y$ in $\bf{A}$ is a cofibration if it belongs to the weakly
saturated class of morphisms generated by $C_0$.
\item[$(W)$]  A morphism $f:X\to Y$ in $\bf{A}$ is a weak equivalence if it belongs to $W$.
\item[$(F)$] A morphism $f:X\to Y$ in $\bf{A}$ is a  fibration if it has the right lifting property with respect to every map which is both a cofibration and a weak equivalence.
\end{itemize}
\end{prop}
\begin{note}
Here \emph{perfectness} is a property of stability under filtered colimits and a generation by a small set $W_0$ (which is more often the intersection of $W$ and the set of maps between presentable objects). The reader can find the exact definition in \cite[Definition A.2.6.10]{Lurie_HTT}.
\end{note}

\begin{warn}
We've used so far the letters $f, g$ as objects of $\M^{\Un}$ so to avoid any confusion we will use $\sigma,\sigma'$ instead.
\end{warn}

Applying the previous proposition we get the following theorem.
\begin{thm}\label{enlarging-msx}
Let $\M$  and $\Ub: \ag \to \M$ be as before. Then  there exists a combinatorial model structure on $\coms$ which is left proper and which may be described as follows. 
\begin{enumerate}
\item A map $\sigma: \F \to \G$ is a weak equivalence if it's an easy weak equivalence i.e, if it's in $\W_{\comse}$.
\item A map $\sigma: \F \to \G$  is a cofibration if it belongs to the weakly
saturated class of morphisms generated by $ \I_{\coms}^{+}$.
\item A morphism $\sigma: \F \to \G$ is a fibration if it has the right lifting property with respect to every map which is both a cofibration and a weak equivalence
\end{enumerate}
We will denote this model category  by $\comsep$. The identity functor $$ \Id: \comse \to \comsep, $$ is a left Quillen equivalence
\end{thm}
\begin{proof}
Condition $(1)$ is straightforward because $\W_{\comse}$ is the class of weak equivalences in the combinatorial model category $\comse$. We also have Condition $(3)$ since a map $\sigma$ in $\I_{\coms}^{+}\tx{-inj}$ is in particular in $\I_{\coms}\tx{-inj}$, therefore it's a trivial fibration in $\comse$ and thus an easy weak equivalence.\\

It remains to check that Condition $(2)$ is also satisfied. Consider the following diagram as in the proposition.
\[
\xy
(0,10)*+{\F}="A";
(20,10)*+{\G}="B";
(0,0)*+{\F'}="C";
(20,0)*+{\G'}="D";
(0,-10)*+{\F''}="X";
(20,-10)*+{\G''}="Y";
{\ar@{->}^-{\sigma}"A";"B"};
{\ar@{->}_-{}"A";"C"};
{\ar@{->}^-{}"B";"D"};
{\ar@{->}^-{}"C";"D"};
{\ar@{->}^-{}"X";"Y"};
{\ar@{->}^-{\theta}"C";"X"};
{\ar@{->}^-{\theta'}"D";"Y"};
\endxy
\]

 If $\sigma: \F \to \G$ is in $\I_{\coms}^{+}$, as mentioned above each component 
$$\sigma_{0}: \F_0\to \G_0$$ 
is a cofibration in $\M$. Following  Lemma \ref{lem-limit-colimit} the source-functor preserve pushouts in $\coms$. It follows that the top components in that diagram are obtained by pushout in $\M$ and the rest follows from Lemma \ref{lem-pushout-important}.\\

A cofibration in $\comse$ is a cofibration in $\comsep$ and since we have the same weak equivalences  then the identity functor is a Quillen equivalence.
\end{proof}

\begin{cor}\label{kb-cell-complex}
Any $\kb(\I)$-cell complex is a cofibration in $\comsep$.
\end{cor}

\section{Producing QS-objects}
The present discussion is also inspired by the theory of co-Segal structures (see \cite{Bacard_lpcosegx}). Let's consider a gain our set $\kb(\I)$. Then following Remark \ref{rmk-projec-com} we know that for any $\sigma= \Gamma(\av) \in \kb(\I)$, the image $|\sigma|$ in $\ag$ is the identity $\Fb(V) \xrightarrow{\Id} \Fb(V)$.\\

Now since $\Gamma$ is a left adjoint, it certainly distributes over coproducts. It follows that if $\theta= \sqcup_{J} \sigma_j$ is a coproduct of maps in $\kb(\I)$, then the image $|\theta|$ is a coproduct of identities $\sqcup_J \Id_{\Fb(V)_j}$, which is an isomorphism. 

This tells us what happens when we use the small object argument to produce a $\kb(\I)$-injective replacement functor.  It takes an object $\F: \F_0 \to \F_1$ and factors it as cofibration followed by a trivial fibration, without changing $\F_1$. This factorization will be regarded as the composition of a $\kb(\I)$-cell complex followed by a weak equivalence.\\

 We refer the reader to \cite{Dwyer_Spalinski}, \cite{Hov-model} for a detailed account on the small object argument. We summarize the previous description as a proposition.

\begin{pdef}\label{pdef-deux-constant}
Let $\Sim: \coms \to \coms$ be the $\kb(\I)$-injective replacement functor obtained by the small object argument. Denote by $\tau: Id \to \Sim$ the induced natural transformation.\\

\begin{enumerate}
\item For every $\F \in \coms$,  $\Sim(\F)$ is a $\Ub$-QS-object
\item For every $\F$, we have an isomorphism  $\Sim(\F)_1 \cong \F_1$.
\item For every $\F$,  the map in $\M$ underlying  $\Sim(\F)$ defines a morphism  $\Sim(\F) \to \iota(\F_1)$ in  $\coms$, which is a trivial fibration and in particular, an easy weak equivalence in $\coms$.
\item If we regard $\F: \F_0 \to \F_1$ as a morphism $ \F \to \iota(\F_1)$ in $\coms$, then we have a factorization of this morphism as a $\kb(\I)$-cell complex followed by a weak equivalence in $\comse$ (whence in $\comsep$):
 $$  \F \xrightarrow{\F} \iota(\F_1)= \F \xrightarrow{\tau_\F} \Sim(\F) \xrightarrow{} \iota(\F_1). $$
\item Let $L: \coms\to  \B$ be a functor that sends easy weak equivalences to isomorphisms and takes any $\kb(\I)$-cell complex to an isomorphism. Then for all $\F \in \coms$,  the image by $L$ of the unit $\eta: \F \to \iota(|\F|)$ is an isomorphism in $\B$. 
\end{enumerate}
\end{pdef}

\section{Bousfield localizations}\label{sec-bousfield}
\begin{warn}
We would like to warn the reader about our upcoming notation for the left Bousfield localizations. We choose to include a small letter $\mathbf{c}$ (for ``correct'') as a superscript in both $\comsep$ and $\comse$ to mean that we are taking the left Bousfield localization with respect to the (same) set $\kb(\I)$. A more suggestive and standard notation should be $\Lb^{\kb(\I)}\comse$ or $\kb(\I)^{-1}\comse$, but as the reader can see, this is too heavy to work with.\\

Instead we will use the notation $\comsec$ and $\comsepc$.
\end{warn}

\begin{rmk}

\begin{enumerate}
\item Since we have the same class of weak equivalences in $\comse$ and in $\comsep$, we have an equivalence of function complexes on these model structures. It follows that a map $\sigma$ in $\coms$ is a $\kb(\I)$-local equivalence in the model structure $\comse$ if and only if it's a $\kb(\I)$-local equivalence in the model structure $\comsep$.
\item A direct consequence of this is that the left Quillen equivalence $$\comse \to \comsep,$$ given by the identity, will pass to a left Quillen equivalence between the respective Bousfield localization:
$$\comsec \to \comsepc.$$ 
\end{enumerate}
\end{rmk}

\subsection{The first localized model category}
We start with the localization of $\comsep$.
\begin{thm}\label{main-thm-1}
Let $\M$ and $\Ub: \ag \to \M$ be as before. Then  there exists a combinatorial model structure on $\coms$ which is left proper and which may be described as follows. 
\begin{enumerate}
\item A map $\sigma: \F \to \G$ is a weak equivalence if and only if it's a $\kb(\I)$-local equivalence.
\item A map $\sigma: \F \to \G$ is cofibration if it's a cofibration in $\comsep$.
\item Any fibrant object $\F$ is a $\Ub$-QS-object. 
\item We will denote this model category  by $\comsepc$. 
\item The identity of $\coms$ determines the universal left Quillen functor
$$\Lb_+: \comsep \to \comsepc.$$

\end{enumerate}

This model structure is the left Bousfield localization of $\comsep$ with the respect to the set $\kb(\I)$. 
\end{thm}

\begin{proof}[Proof of Theorem \ref{main-thm-1}]
The existence of the left Bousfield localization and the left properness is guaranteed by Smith's theorem on left Bousfield localization for combinatorial model categories. We refer the reader to Barwick \cite[Theorem 4.7]{Barwick_localization} for a precise statement. This model structure is again combinatorial.\\

For the rest of the proof we will use the following facts on Bousfield localization and the reader can find them in Hirschhorn's book \cite{Hirsch-model-loc}. 
\begin{enumerate}
\item A weak equivalence in $\comsepc$ is a \emph{$\kb(\I)$-local weak equivalence}; we will refer them as \emph{new weak equivalence}. And any easy weak equivalence (old one) is a new weak equivalence.
\item The new cofibrations are the same as the old ones and therefore the new trivial fibrations are just the old ones too. In particular a trivial fibration in the left Bousfield localization is an easy weak equivalence. 
\item The fibrant objects are the  $\kb(\I)$-local objects that are fibrant in the original model structure.
\item Every map in $\kb(\I)$ becomes a weak equivalence in $\comsec$, therefore an isomorphism in the homotopy category.
\end{enumerate}

Let $\F$ be a fibrant object in $\comsepc$, this means that the unique map $\F \to \ast$ has the RLP with respect to any trivial cofibration. Now observe elements of $\kb(\I)$ are trivial cofibrations in $\comsepc$ because they were old cofibrations and become weak equivalences. So any fibrant $\F$ must be in particular $\kb(\I)$-injective and thanks to  Lemma \ref{k-inj-cosegal}, we know that $\F$ underlies a trivial fibration.
\end{proof}
\begin{cor}
Let $\Sim' $ be a  fibrant replacement in $\comsepc$. Then a map $\sigma: \F \to \G$ is a weak equivalence in $\comsepc$ if and only if the map 
$$\Sim'(\sigma): \Sim'(\F) \to \Sim'(\G),$$
 is a level-wise weak equivalence in $\M^{\Un}$. 
\end{cor}

\begin{proof}
Since $\Sim'$ is a fibrant replacement functor in the new model structure, then by the second assertion of the previous theorem, $\Sim'(\F)$ is a $\Ub$-QS-object  for all $\F$.  

 By the $3$-for-$2$ property of weak equivalences in any model category, a map $\sigma$ is a weak equivalence if and only if $\Sim'(\sigma)$ is a weak equivalence. But $\Sim'(\sigma)$ is a weak equivalence of fibrant objects in the Bousfield localization, therefore it's a weak equivalence in the original model structure.

In the end we see that  $\sigma$ is a weak equivalence in $\comsepc$ if and only if $\Sim'(\sigma)$ is an easy weak equivalence in $\comsep$. Now an easy weak equivalence between $\Ub$-QS-objects is just a level-wise weak equivalence.
\end{proof}
\begin{prop}\label{prop-eta-kx-loc-equiv}
For any $\F \in \coms$, the canonical map 
$\F \to |\F|$ is an equivalence in $\comsepc$ i.e, it's a $\kb(\I)$-local equivalence in $\comsep$ (whence in $\comse$).  
\end{prop}
\begin{proof}
Every element $\Gamma(\av) \in \kb(\I)$ becomes a trivial cofibration  in $\comsepc$, since they were cofibration in $\comsep$. In particular every $\kb(\I)$-cell complex is a trivial cofibration in the left Bousfield localization $\comsepc$.  The proposition follows from  Assertion $(4)$ of Proposition \ref{pdef-deux-constant}.
\end{proof}
\subsection{The second localized model category}
We now localize the original model category $\comse$ that doesn't contain a priori the set $\kb(\I)$ among the class of cofibrations. The theorem we give below is also a straightforward application of Smith's theorem for left proper combinatorial model category.
\begin{thm}\label{main-thm-2}
Let $\M$ and $\Ub: \ag \to \M$ be as before.. Then  there exists a combinatorial model structure on $\coms$ which is left proper and which may be described as follows.
\begin{enumerate}
\item A map $\sigma: \F \to \G$ is a weak equivalence if and only if it's a $\kb(\I)$-local equivalence. 
\item A map $\sigma: \F \to \G$ is cofibration if it's a cofibration in $\comse$.
We will denote this model category  by $\comsec$
\item If the adjunction  $\Ub:\com \leftrightarrows \M: \Fb$ is part of a Quillen adjunction,  then the adjunction $$ |-|: \comse \leftrightarrows \com: \iota, $$ descends to a Quillen adjunction:
$$ |-|^{\mathbf{c}}: \comsec \leftrightarrows \com: \iota.$$
In particular the inclusion $\iota: \com \to \comsec$ is again a right Quillen functor.
\item  The identity of $\coms$ determines the universal left Quillen functor
$$\Lb: \comse \to \comsec.$$
\end{enumerate}
This model structure is the left Bousfield localization of $\comse$ with the respect to the set $\kb(\I)$. 
\end{thm}
\begin{proof}
The existence and characterization of the Bousfield localization follows also from Smith theorem. The first two assertions are clear. Assertion $(3)$ is a consequence of the universal property of the left Bousfield localization. Indeed we have a left Quillen functor $|-|: \comse \to \com$ that takes elements of $\kb(\I)$ to isomorphisms in $\com$ (Remark \ref{rmk-projec-com}).

 Therefore there exists a unique left Quillen functor $|-|^{\mathbf{c}}: \comsec \to \com$ such that we have an equality
$$|-|: \comse \to \com= \comse \xrightarrow{\Lb} \comsec \xrightarrow{|-|^{\mathbf{c}}} \com.$$
\end{proof}

\section{The Quillen equivalence}\label{sec-Quillen-equiv}
Let's now consider a specific situation where $\Ub:\com \leftrightarrows \M: \Fb$ is a Quillen adjunction in which $\Ub$ preserves and reflects the weak equivalences. A typical situation is of a right-induced model structure on $\com$, such as a category of algebras with the free-forgetful adjunction.\\

We start with the following lemma which is useful to establish the Quillen equivalence.
\begin{lem}\label{lem-reflect-equiv}
Let $\M$ be as before.  Assume that $\Ub: \com \to \M$ is a right Quillen functor that preserves and reflects the weak equivalences.   Let $\sigma : \C \to \D$ be a morphism in $\com$ regarded also as morphism in $\coms$. 

Then $\sigma$ is a $\kb(\I)$-local equivalence in $\comse$ (whence in $\comsep$) if and only if it's a weak equivalence in $\com$. 
\end{lem}
\begin{proof}
The if part is clear since a weak equivalence in $\com$ is an easy weak equivalence in $\coms$ and therefore it's also a weak equivalence in the Bousfield localization. But the weak equivalences in the Bousfield localization are precisely the $\kb(\I)$-local equivalences.\\

Let's now assume that $\sigma: \C \to \D$ becomes a $\kb(\I)$-local equivalence.\\

Use the axiom of factorization in the model category $\com$ to factor $\sigma$ as trivial cofibration followed by a fibration: 
$$\sigma= \C \xhookrightarrow[\sim]{\sigma_1} \Ea \xtwoheadrightarrow{\sigma_2} \D.$$  

Since we know that inclusion $\iota: \com \to \comse$ is a right Quillen functor when we pass to the left Bousfield localization $\comsec$, it follows that the map $\sigma_2: \Ea \to \D$ is a fibration in $\comsec$. Now  as $\sigma_1$ is a weak equivalence in $\com$, by the if part, it's also a $\kb(\I)$-local equivalence. \\

By the $3$-for-$2$ property of $\kb(\I)$-local equivalences applied to the equality $\sigma= \sigma_2 \circ \sigma_1$, we find that $\sigma_2$ is also a $\kb(\I)$-local equivalence. \\

In the end we see that $\sigma_2$ is simultaneously a fibration in the Bousfield localization $\comsec$ and a $\kb(\I)$-local equivalence, therefore it's a trivial fibration in $\comsec$. But a trivial fibration in this left Bousfield localization is the same as a trivial fibration in the original model structure. This means that $\sigma_2$ is usual trivial fibration, in particular a weak equivalence in $\M$. And since $\Ub$ reflects the weak equivalences, it follows that $\sigma_2$ is a weak equivalence in $\com$. 

Then $\sigma= \sigma_2 \circ \sigma$ is a composite of  weak equivalences in $\com$ and the lemma follows.
\end{proof}

\subsection{The main Theorem}
\begin{thm}\label{main-thm-quillen-equiv}
Let $\M$ be combinatorial and left proper model category. Let $\Ub: \com \to \M$ be a right adjoint between locally presentable categories that is faithful.  Then the following hold. 
\begin{enumerate}
\item The left Quillen equivalence $\comse \to \comsep$ induced by the identity of $\coms$ descends to a left Quillen equivalence between the respective left Bousfield localizations with respect to $\kb(\I)$:
$$\comsec \to \comsepc.$$ 
\item Assume that $\Ub: \com \to \M$ is a right Quillen functor that preserves and reflects the weak equivalences. Then the adjunction
$$ |-|^{\mathbf{c}}: \comsec \leftrightarrows \com: \iota,$$
is a Quillen equivalence. 
\item The diagram $\comsepc \xleftarrow{}\comsec \xrightarrow{|-|^{\mathbf{c}}} \com$ is a zigzag of Quillen equivalences. In particular we have a diagram of equivalences between the homotopy categories.
$$\Ho[\comsepc] \xleftarrow{\simeq} \Ho[\comsec] \xrightarrow{\simeq} \Ho[\com].$$
\end{enumerate}
\end{thm}
\begin{proof}
We only need to prove the second assertion, namely that we have a Quillen equivalence
$$ |-|^{\mathbf{c}}: \comsec \leftrightarrows \com: \iota.$$

For this it suffices to show that if $\F \in \comsec$ is cofibrant and if $\C\in \com$ is fibrant, then a map $\sigma:\F \to \iota(\C)$ is a weak equivalence in $\comsec$ if and only if the adjunct map $\ol{\sigma}:|\F| \to \C$ is a weak equivalence in $\com$.\\

Since the inclusion $\iota: \com \to \comsec$ is fully faithful, we will identity $\C$ with $\iota(\C)$ and $|\F|$ with $\iota(|\F|)$. Let $\eta: \F \to |\F|$ be the unit of the adjunction. Then all three maps fit in a commutative triangle in $\comsec$:
\begin{equation}\label{triang}
\xy
(0,15)*+{\F}="W";
(0,0)*+{|\F| }="X";
(20,0)*+{\C}="Y";
{\ar@{->}^-{\ol{\sigma}}"X";"Y"};
{\ar@{->}^-{\eta}"W";"X"};
{\ar@{->}^-{\sigma}"W";"Y"};
\endxy
\end{equation}

Thanks to Proposition \ref{prop-eta-kx-loc-equiv} we know that $\eta: \F \to |\F|$ is always a $\kb(\I)$-local equivalence. Then by $3$-for-$2$,  we see that $\sigma$ is a $\kb(\I)$-local equivalence if and only if $\ol{\sigma}$ is. Now thanks to Lemma \ref{lem-reflect-equiv} we know that $\ol{\sigma}$ is a $\kb(\I)$-local equivalence if and only if it's an equivalence in $\com$ and the theorem follows.
\end{proof}

\begin{note}
If $\M$ is tractable i.e., the domain of every element in $\I$ is cofibrant, then it's possible to show that the model category $\comsec$ is good enough. Using the Homotopy Extension Lifting Property (HELP), it's possible to show that every fibrant object in $\comsec$ is a $\Ub$-QS-object.
\end{note}

\section{A remark on the left properness for algebras}

Let's consider the original situation of the homotopy theory of algebras in $\M$. It's a well know fact that the category $\oalg(\M)$ is not always left proper (see for example \cite{Rezk_model}, \cite{Muro_hop2}).  A concrete example of the failure of left properness for monoids was communicated to the author by Fernando Muro. The discussion that follows is motivated by a question that came out after some communications with him.\\

We also wanted to understand the recent paper of Batanin-Berger \cite{Batanin_Berger_poly} on the left properness for algebras over a polynomial monad.  We consider the category of algebras over a monad or an operad $\O$, with the projective model structure. We remind the reader that $\M$ is still assumed to be combinatorial and left proper.

\begin{prop}\label{prop-left-prop}
Let $\com=\oalg(\M)$ be the category of algebras  equipped with the free-forgetful Quillen adjunction $\Ub: \oalg(\M) \leftrightarrows \M: \Fb$.  \\

let $D$ be a pushout square in the category $\oalg(\M)$
 \[
 \xy
(0,18)*+{\F}="W";
(0,0)*+{ \G}="X";
(30,0)*+{\Ba}="Y";
(30,18)*+{\Aa}="E";
{\ar@{->}^-{\ol{\sigma}}"X";"Y"};
{\ar@{->}^-{\theta}"W";"X"};
{\ar@{->}^-{\sigma}"W";"E"};
{\ar@{->}^-{\varepsilon}"E";"Y"};
\endxy
\] 

Assume that $\sigma: \F \to \Aa$ is a weak equivalence of algebras. If the image $\Ub(\theta)$ is a cofibration in $\M$, then the map $\ol{\sigma}$ is a weak equivalence of algebras.
\end{prop}

\begin{proof}
Let's embed this pushout square in $\coms$ through the inclusion $\iota: \com \hookrightarrow \coms$. This inclusion is a right adjoint, so it doesn't always preserve pushouts. We know from Lemma \ref{lem-limit-colimit} and Lemma \ref{lem-pushout-important}, how we compute colimits in $\coms$, and in particular pushouts.\\

Furthermore, the target functor $|-|: \coms \to \com$ is left adjoint to the previous embedding. And as a left adjoint it preserves any kind of colimits, including pushouts. So if we compute the colimit in $\coms$ and then project it to $\com$ we get the same thing.\\

The pushout object in the category $\coms$ is an object $\Ea$ that corresponds to the canonical morphism $$\Ea= \{[\Ub(\G) \cup^{\Ub(\F)} \Ub(\Aa)] \to \Ub(\Ba) \}.$$ This simply tells us what we do when we compute a pushout: first we take the pushout of underlying objects and then we modify it by imposing equations (taking quotients) to get an algebra. \\

Now since $\M$ is left proper we know that the canonical map between underlying objects is a weak equivalence in $\M$:
$$\Ub(\G) \to  [\Ub(\G) \cup^{\Ub(\F)} \Ub(\Aa)].$$ 

This map together with the map $\ol{\sigma}$ determine a map $\delta: \iota(\G) \to \Ea$ in $\coms$ that is an easy weak equivalence and therefore a $\kb(\I)$-local weak equivalence.\\

On the other hand, as mentioned before the map that underlies the object $\Ea$ defines a tautological morphism 
$$\eta: \Ea \to \iota(\Ba).$$

Thanks to Proposition \ref{prop-eta-kx-loc-equiv} we know that this map is always a $\kb(\I)$-local weak equivalence. But this map fits in a factorization of $\ol{\sigma}$ as:
$$\iota(\sigma)= \eta \circ \delta.$$

We see that $\iota(\sigma)$ is the composite of two $\kb(\I)$-local equivalence, therefore it's a  $\kb(\I)$-local equivalence. But thanks to Lemma \ref{lem-reflect-equiv}, this means that $\ol{\sigma}$ is a weak equivalence of algebras as claimed.
\end{proof}

As a corollary we get:

\begin{cor}\label{left-prop-alg}
If in the projective model structure on $\oalg(\M)$, the image $\Ub(\theta)$ of every cofibration  $\theta$ is a cofibration in $\M$, then $\oalg(\M)$ is also left proper.
\end{cor}

Following this corollary, it turns out that the left properness for algebras is intimately related to the monad or operad $\O$. By a lemma of Muro (\cite[Lemma 6.8]{Muro_hop2}), if the operad $\O$ is nice enough, every cofibration in $\oalg(\M)$ has an underlying cofibration. In this case we recover Muro's result about the left properness of the category $\oalg(\M)$ (\cite[Cor. 1.14]{Muro_hop2}).\\

In \cite{EKMM_book}, the authors consider a \emph{Cofibration Hypothesis}. Under this hypothesis, it's proved in \cite[Theorem 4.14]{EKMM_book} that the hypothesis of Corollary \ref{left-prop-alg} holds.\\

Finally, the definition of \emph{admissible monad} of Batanin-Berger \cite[Def. 2.9]{Batanin_Berger_poly}, implies the hypothesis of Corollary \ref{left-prop-alg}.

\bibliographystyle{plain}
\bibliography{Bibliography_LP_COSEG}
\end{document}